\documentclass[12pt]{amsart}
\usepackage{amscd,amsmath,amsthm,amssymb}
\usepackage{color}
\usepackage{amsfonts,amssymb,amscd,amsmath,enumerate,verbatim,enumitem}
\usepackage{lineno}
 \def\NZQ{\mathbb}               
 \def\NN{{\NZQ N}}
 
 \def\ZZ{{\NZQ Z}}

 \def\frk{\mathfrak}               

 \def\mm{{\frk m}}
 
 \def\nn{{\frk n}}
 \def\opn#1#2{\def#1{\operatorname{#2}}} 
 \opn\chara{char} \opn\length{\ell} \opn\pd{pd} \opn\rk{rk}
 \opn\projdim{proj\,dim} \opn\injdim{inj\,dim} \opn\rank{rank}
 \opn\depth{depth} \opn\grade{grade} \opn\height{height}
 \opn\embdim{emb\,dim} \opn\codim{codim}
 
 \opn\Tr{Tr} \opn\bigrank{big\,rank}
 \opn\superheight{superheight}\opn\lcm{lcm}
 \opn\trdeg{tr\,deg}
 \opn\reg{reg} \opn\lreg{lreg} \opn\ini{in} \opn\lpd{lpd}
 \opn\size{size} \opn\sdepth{sdepth}
 \opn\link{link}\opn\fdepth{fdepth}\opn\lex{lex}
 \opn\div{div} \opn\Div{Div} \opn\cl{cl} \opn\Cl{Cl}
 \opn\Spec{Spec} \opn\Supp{Supp} \opn\supp{supp} \opn\Sing{Sing}
 \opn\Ass{Ass} \opn\Min{Min}\opn\Mon{Mon}
 \opn\Ann{Ann} \opn\Rad{Rad} \opn\Soc{Soc}
 \opn\Im{Im} \opn\Ker{Ker} \opn\Coker{Coker} \opn\Am{Am}
 \opn\Hom{Hom} \opn\Tor{Tor} \opn\Ext{Ext} \opn\End{End}
 \opn\Aut{Aut} \opn\id{id}
 
 \opn\nat{nat}
 \opn\pff{pf}
 \opn\Pf{Pf} \opn\GL{GL} \opn\SL{SL} \opn\mod{mod} \opn\ord{ord}
 \opn\Gin{Gin} \opn\Hilb{Hilb}\opn\sort{sort}
 \opn\aff{aff} \opn
 \con{conv} \opn\relint{relint} \opn\st{st}
 \opn\lk{lk} \opn\cn{cn} \opn\core{core} \opn\vol{vol}  \opn\inp{inp} \opn\nilpot{nilpot}
 \opn\link{link} \opn\star{star}\opn\lex{lex}\opn\set{set}
 \opn\width{wd}
 \opn\gr{gr}
 
 \def\pot#1#2{#1[\kern-0.28ex[#2]\kern-0.28ex]}

 \opn\dirlim{\underrightarrow{\lim}}
 \opn\inivlim{\underleftarrow{\lim}}
 \let\to=\rightarrow
 
 \def\Implies{\ifmmode\Longrightarrow \else
         \unskip${}\Longrightarrow{}$\ignorespaces\fi}
 \def\implies{\ifmmode\Rightarrow \else
         \unskip${}\Rightarrow{}$\ignorespaces\fi}
 \def\iff{\ifmmode\Longleftrightarrow \else
         \unskip${}\Longleftrightarrow{}$\ignorespaces\fi}

 \let\:=\colon
 \newtheorem{Theorem}{Theorem}[section]
 \newtheorem{Lemma}[Theorem]{Lemma}
 \newtheorem{Corollary}[Theorem]{Corollary}
 \newtheorem{Proposition}[Theorem]{Proposition}
 \newtheorem{Remark}[Theorem]{Remark}
 
 \newtheorem{Example}[Theorem]{Example}
 
 \newtheorem{Definition}[Theorem]{Definition}

 \newtheorem{Question}[Theorem]{Question}
 \let\epsilon\varepsilon
 \let\kappa=\varkappa
 \def\qed{\ifhmode\textqed\fi
       \ifmmode\ifinner\quad\qedsymbol\else\dispqed\fi\fi}
 \def\textqed{\unskip\nobreak\penalty50
        \hskip2em\hbox{}\nobreak\hfil\qedsymbol
        \parfillskip=0pt \finalhyphendemerits=0}
 \def\dispqed{\rlap{\qquad\qedsymbol}}
 
 \opn\dis{dis}
 \def\pnt{{\raise0.5mm\hbox{\large\bf.}}}
 
 \opn\Lex{Lex}

\begin{document}
\title {Quadratic numerical semigroups and the Koszul property}

\author {J\"urgen Herzog, Dumitru I.\ Stamate}

\address{J\"urgen Herzog, Fachbereich Mathematik, Universit\"at Duisburg-Essen, Campus Essen, 45117
Essen, Germany} \email{juergen.herzog@uni-essen.de}

\address{Dumitru I. Stamate, Faculty of Mathematics and Computer Science, University of Bucharest, Str. Academiei 14, Bucharest, Romania, and  \newline  \indent
Simion Stoilow Institute of Mathematics of the Romanian Academy, Research group
of the project PN-II-RU-PD-2012-3-0656, P.O.Box 1-764, Bucharest 014700, Romania}
\email{dumitru.stamate@fmi.unibuc.ro}

\dedicatory{ }

\begin{abstract}
Let $H$ be a numerical semigroup. We give effective bounds for the multiplicity $e(H)$ when the associated graded ring $\gr_\mm K[H]$ is defined by quadrics.
We classify Koszul complete intersection semigroups in terms of gluings.
Furthermore, for several classes of numerical semigroups considered in the literature
(arithmetic, compound, special almost complete intersections, $3$-semigroups, symmetric or pseudo-symmetric $4$-semigroups) we classify those which are Koszul.
\end{abstract}

\thanks{}


\maketitle

\section*{Introduction}

Let $K$ be a field. A standard graded $K$-algebra $R$ with graded maximal ideal $\mm$ is called {\em Koszul} if the $R$-module $K\cong R/\mm$
has an $R$-linear resolution. It is known that if $I$, the defining ideal of $R$ has a Gr\"obner basis of quadrics, then $R$ is Koszul,
and also that if $R$ is Koszul, then $I$ is generated by quadrics.
Although it is in general difficult to certify that an algebra is Koszul, the   properties of this class of rings make  it an interesting  endeavour.
We refer to  the survey articles \cite{Froberg-fez}  and \cite{CdNR} for more details.

Due to the promise of a rich theory, it is of interest to study the Koszul property for a larger class of rings. Inspired by an idea of Fr\"oberg \cite{Fro},
in \cite{HRW} the first author, Reiner and Welker consider the Koszul property for the associated graded ring of an affine semigroup ring,
with respect to the maximal multigraded ideal.  For instance, it is proved that for a $2$-dimensional normal affine semigroup ring  its associated graded ring is Koszul,
see \cite[Proposition 5.3]{HRW}.

In this paper we focus on  the case of $1$-dimensional affine semigroup rings, i.e. those coming from numerical semigroups.
Recall that a  numerical semigroup $H$ is  a subset of the nonnegative integers that is closed under addition, contains $0$ and $\NN \setminus H$ is finite or,
equivalently, the gcd of all elements in $H$ equals $1$.
We denote by $G(H)$ the unique minimal system of generators for $H$. The multiplicity and the embedding dimension of $H$
are defined as $e(H)= \min G(H)$ and $\embdim (H)= |G(H)|$, respectively.
If $n= \embdim (H)$ we   say that $H$ is an $n$-semigroup. We denote $K[H]=\oplus_{h\in H} Kt^h \subset K[t]$ the semigroup ring associated to $H$.

The tangent cone of $K[H]$ is the associated graded ring $\gr_\mm K[H]=\oplus_{i\geq 0} \mm^i/\mm^{i+1}$  with respect to the maximal ideal $\mm=(t^h: h \in H, h \neq 0)K[H]$.

If $G(H)= \{ a_1, \dots, a_n \}$, the toric ideal $I_H$ is defined as the kernel of the $K$-algebra map 
$\phi: K[x_1, \dots, x_n] \to K[H]$ letting $\phi(x_i)=t^{a_i}$ for $i=1, \dots, n$.
It is known  that $I_H$ is generated by the binomials $f=\prod_{i=1}^n x_i ^{\alpha_i}- \prod_{i=1}^n x_i^{\beta_i}$
where $\alpha_i, \beta_i \geq 0$ for all $i=1,\dots, n$ and $\sum_{i=1}^{n} \alpha_i a_i= \sum_{i=1}^{n} \beta_i a_i$.
It is enough to use only such binomials where $\alpha_i \beta_i=0$ for all $i=1, \dots, n$.

For a nonzero polynomial $f$   its {\em initial form} $f^*$ is the homogeneous component of $f$ of least degree and the initial degree of $f$ is defined as $\deg f^*$.
For an ideal $I$ we let $I^*= (f^*: f\in I, f\neq 0)$. A {\em standard basis} for $I $ is a set of polynomials in $I $ whose initial forms generate $I^*$.
It is known and easy to see that a standard basis is also a generating set for $I$.

With this notation, one can check that  $\gr_\mm K[H] \cong K[x_1,\dots, x_n]/I^*_H$.
From the algorithms that may be used to compute  $I_H^*$ (see \cite{Eis} or \cite{EH}) one gets that $I_H^*$
is generated by  monomials and possibly homogeneous binomials.

\medskip

We are interested in numerical semigroups $H$ such that $\gr_\mm K[H]$ is Koszul.
In general, even if $I_H^*$ is quadratic, the   tangent cone $R=\gr_\mm K[H]$ may not be Koszul. 
For instance, one can check with Singular (\cite{Sing}) that for $H= \langle 12,14, 15, 16, 18, 19 \rangle$,
  $I_H^*$ is generated in degree 2 and $\beta^{R}_{4,5}(K) =1$, hence the resolution of $K$ over $R$ is not linear.

We say that $H$ is a {\em Koszul, quadratic}, or {\em G-quadratic semigroup}
if $\gr_\mm K[H]$ is a Koszul ring, respectively $I_H^*$ is generated in degree 2 or
(possibly after a suitable change of coordinates) it has a Gr\"obner basis of quadrics with respect to some term order.
Note that by  \cite{RS}, the quadratic property  depends only on the generators  of the semigroup and it  does not depend  on the field $K$.
When discussing the Koszul property of $H$ we work over a fixed field $K$, although we do not know of any semigroup $H$ where the Koszul property depends on the field of coefficients.

For some of our arguments to work we need to assume that the field $K$ is infinite.

The ideal $I$  is called a {\em complete intersection ideal} (CI for short) if it is minimally generated by $\height I$ elements.
In case $\mu(I)=1+\height I$, one says that $I$ is an {\em almost complete intersection ideal}.
We say that a numerical semigroup $H$ is an (almost) complete intersection if $I_H$ has that property. Note that in general, if
$I_H$ is an (almost) complete intersection ideal, that property may no longer hold for $I^*_H$. 
However, if $I_H^*$ is generated in degree $2$, we prove in Lemma \ref{lemma:quad-basis} that  $I_H$ is almost CI if and only if $I_H^*$ is of the same kind.

\medskip

Let $n=\embdim(H)$. It is easy to see that  $n\leq e(H)$.  In Section \ref{sec:bounds}
we show that if $H$ is quadratic, there is also an upper bound, namely $e(H)\leq 2^{n-1}$.
It is shown that if either one of these bounds is reached, then $H$ is a Koszul semigroup.
The upper bound is reached if and only if $I_H^*$ is generated by a regular sequence of quadrics.
These results are valid more generally for $1$-dimensional Cohen-Macaulay local rings with infinite residue  field, and our proofs are given in this generality.

Numerical experiments with Singular (\cite{Sing}) make us believe that not all the values in the interval $[n, 2^{n-1}]$
are possible for the multiplicity of a quadratic semigroup $H$.
If $\gr_\mm K[H]$ is not a complete intersection, under the extra assumption that $\gr_\mm K[H]$ is Cohen-Macaulay,
in Theorem \ref{thm:forbidden} we prove that
$e(H) \leq 2^{n-1}-2^{n-3}$, and if equality holds then $H$ is $G$-quadratic and it is an almost complete intersection semigroup.

Interesting classes of semigroups arise from semigroups of smaller embedding dimension by the so-called {\em gluing} construction.
If $H_1, H_2$ are numerical semigroups and $c_1 , c_2$ are coprime integers such that
$c_1 \in H_2\setminus G(H_2)$ and $c_2\in H_1\setminus G(H_1)$,  we say that
$H=\langle c_1H_1, c_2 H_2\rangle$ is obtained by gluing $H_1$ and $H_2$.
The most prominent result in this direction is Delorme's characterization of complete intersection numerical semigroups (\cite{Delorme}). Namely,
any such semigroup is obtained by a sequence of gluings starting from $\NN$, see Theorem \ref{thm:Delorme}.

If $c_1=2$ and $H_2=\NN$ we say that $H= \langle 2 H_1,  c_2  \rangle$ is obtained from $H_1$ by a  {\em quadratic gluing}.
As a main result  of Section \ref{sec:gluing} we complement Delorme's theorem by showing that any quadratic complete intersection
numerical semigroup is obtained by a sequence of quadratic gluings, see Theorem~\ref{thm:quad-ci}.
This result is a consequence of Theorem \ref{thm:glue-quadratic} and Corollary \ref{cor:glue-2} where we show that
the semigroup $H=\langle 2 L, \ell \rangle$ is quadratic, Koszul or G-quadratic if and only if $L$ has the respective property.

In Section \ref{sec:examples} we apply the methods described so far to study the occurrence of quadratic or Koszul members
in several important families of numerical semigroups for which the defining equations of the toric ring are  better understood.
In Propositions \ref{prop:arithmetic} and \ref{prop:compound} we show that
the multiplicity of a quadratic semigroup generated by an arithmetic, respectively by a geometric sequence is either very small, compared to the embedding dimension,
or, respectively, as large as it is allowed by Theorem \ref{thm:bounds}.

This extremal property resembles another extremal behaviour for these classes of  semigroups. When $H$ is generated by a geometric sequence,
the Betti numbers in the resolution of the tangent cone $\gr_\mm K[H]$ are the smallest possible fixing the embedding dimension
(because now  $\gr_\mm K[H]$ is a complete intersection, see Proposition \ref{prop:compound}).
In previous joint work we conjectured that for a given width of $H$,
the largest Betti numbers for   $\gr_\mm K[H]$ are obtained by some arithmetic sequences, see \cite[Conjecture 2.1]{HeS}.

In the rest of Section \ref{sec:examples} we describe completely the quadratic $3$-generated and the quadratic symmetric or pseudo-symmetric $4$-semigroups.
We refer to Subsection \ref{ssec:symmetric} for the exact definitions. It is worth mentioning that these quadratic semigroups
are also G-quadratic.

\section{Bounds for the multiplicity}
\label{sec:bounds}

In this section we present some restrictions  for the multiplicity of a  quadratic numerical semigroup.

\begin{Theorem}
\label{thm:bounds}
Let $H$ be a quadratic numerical semigroup minimally generated by $n>1$  elements, and let  $K[H]$ be  its semigroup ring.  Then
\begin{enumerate}
\item[{\em(a)}] $n\leq e(H)\leq 2^{n-1}$;
\item[{\em (b)}] $e(H)=n$  \iff  $I^*_H$ has a linear resolution;
\item[{\em (c)}] $e(H)=2^{n-1}$ \iff  $I^*_H$ is a CI ideal \iff  $I_H$ is a CI ideal.
\end{enumerate}
\end{Theorem}

More generally, this theorem, formulated for semigroup rings, is true for any $1$-dimensional local Cohen-Macaulay ring $(A,\mm)$ with  a presentation
$A=B/I$ where $(B,\nn)$ is a  regular local ring with infinite residue field $S/\nn$, and where $I\subseteq \nn^2$.
The next sequence of propositions shows this result in this generality.

Let $\widehat{K[H]}$ be the local ring obtained as the $\mm$-adic completion of $K[H]$.
Theorem ~\ref{thm:bounds} follows  from   the fact  that $\gr_\mm K[H] \cong \gr_\mm \widehat{K[H]}$
and  $e(H)$ coincides with the multiplicity of $\widehat{K[H]}$.

Let $R=\gr_\mm A$.  Then $R\cong  S/I^*$, where $S=\gr_\nn B$ is a polynomial ring and $I^*$ is the ideal of initial forms of  $I$.
We say that $A$ is {\em quadratic} if $I^*$ is generated by quadrics.

\begin{Proposition}
\label{prop:local-ineq}
If $A$ is quadratic, then
$$
\embdim A \leq e(A) \leq 2^{\embdim A -1}.
$$
If $e(A)=\embdim(A)$ we say that $A$ has {\em minimal multiplicity}.
\end{Proposition}

\begin{proof}
Since  $A$ is Cohen-Macaulay and its residue field $K=A/\mm$ is infinite, a classical result of Abhyankar \cite{A} gives that  $e(A) \geq \embdim (A) - \dim A +1= \embdim(A)$.

Let $n=\embdim(R)$.
Since $K$ is infinite, there exists  $x \in R_1$ such that $\ell (0:x) < \infty$, see  \cite[Lemma 4.3.1]{HH-monomials}. Denote $\bar{R}=R/(x)$.
Then $H_R(t)= Q(t)/(1-t)$ with $Q(1) =e(R)$. From the exact sequence
\begin{eqnarray*}
0  \rightarrow  (0:_R x)  \rightarrow R(-d) \stackrel{x}{\rightarrow}  R \rightarrow \bar{R} \rightarrow 0
\end{eqnarray*}
we obtain that
$$
H_{\bar{R}}(t)= (1-t)H_R(t) + H_L(t)= Q(t)+ H_L(t).
 $$
This yields
\begin{equation}
\label{eq:rbar}
e(R) = Q(1) \leq Q(1)+H_L(1) =H_{\bar{R}}(1)=    e(\bar{R}).
\end{equation}

Since $R$ is quadratic, we get that $\bar{R} \cong \bar{S}/J$, where $\bar{S}$ is a polynomial ring in $n-1$
variables and where $J$ is generated by quadrics.
As $K$ is infinite, $J$ contains a regular sequence  $q_1,\dots, q_{n-1}$ of quadrics.
It follows that
\begin{equation}
\label{eq:q-ci}
e(R) \leq e(\bar{R}) \leq e(\bar{S}/(q_1, \dots, q_{n-1}))=2^{n-1}.
\end{equation}
\end{proof}

\begin{Proposition}
\label{prop:linear}
The local ring $A$ has minimal multiplicity if and only if  $I^*$ has a    $2$-linear $S$-resolution.
\end{Proposition}

\begin{proof}
As in the proof of Proposition \ref{prop:local-ineq} we denote $\bar{R}=R/(x)$ for some $x\in R_1$ with $\ell (0:x) < \infty$,
and we write $\bar{R}=\bar{S}/J$.
We use the fact that an $\bar{\mm}$-primary ideal in $\bar{S}$ has a linear resolution if and only if it is
a power of the graded maximal ideal $\bar{\mm}$ of $\bar{S}$, see \cite[Exercise 4.1.17(b)]{BH}.

We denote $n=\embdim(A)$.

Suppose that $e(A)=n$. Then  $R$ is Cohen-Macaulay by a result of J.~Sally, see \cite[Theorem 2]{Sally-cm}.
This implies that $x$ is regular on $R$, hence $e(\bar{R})=e(R)$.
This is only possible if $J= \bar{\mm}^2$, which has  a $2$-linear resolution over $\bar{S}$ by the remark before.
Since $x$ is regular on $R$, it follows that $I^*$ itself is quadratic and it has a linear resolution over $S$.

Conversely, assume that $I^*$ has a    $2$-linear $S$-resolution. Therefore $I^*$ and $J$ are generated by quadrics.
If  $e(R)>n$, by \eqref{eq:rbar} we get  $e(\bar{R}) >n$, which implies that $J \subsetneq \bar{\mm}^2$. Therefore, $\reg \bar{R} >1$.
By \cite[Proposition 20.20]{Eis}
$$
\reg R = \max \{ \reg (0:_R x), \reg \bar{R} \}  >1,
$$
hence $I^*$ does not have a linear resolution over $S$, a contradiction.
\end{proof}

\begin{Proposition}
\label{prop:ci-local}
Assume $A$ is quadratic.
The following statements are equivalent:
\begin{enumerate}
\item[{\em (a)}] $e(A)= 2^{\embdim A-1}$;
\item[{\em (b)}] $I^*$ is a complete intersection ideal;
\item[{\em (c)}] $I$  is a complete intersection ideal.
\end{enumerate}
\end{Proposition}

\begin{proof}
(a) \implies (b):  Since $e(A)=2^{\embdim A-1}$, by \eqref{eq:q-ci} it follows that $J$ is generated by a regular sequence of quadrics and  $e(R)=e(\bar{R})$.
The latter implies that $x$ is regular on $R$, therefore $I^*$ is generated by a regular sequence.

(b) \implies (a):  follows from the fact that $I$ is generated by a regular sequence of $n-1$ quadrics.

The equivalence of (b) and (c) is a consequence of Lemma \ref{lemma:quad-basis}.
\end{proof}

\begin{Lemma}
\label{lemma:quad-basis}
Let $(B,\nn)$ be a regular local ring and $I\subseteq \nn^2$ any ideal such that $I^*$ is generated in degree $2$.
\begin{enumerate}
\item[{\em (a)}] Let $\mathcal{F}\subset I$ be a finite set. Then $\mathcal{F}$ is a (minimal) standard basis for $I$ if and only if  it is a (minimal) generating set for $I$.
\item[{\em (b)}] The ideal $I$ is  an (almost) complete intersection ideal  if and only if $I^*$ is an (almost) complete intersection ideal.
\end{enumerate}
\end{Lemma}

\begin{proof}
 (a) Let   $\mathcal{F}= \{ f_1, \dots, f_r \}$ be a minimal standard basis for $I$.
As a general fact, $\mathcal{F}$ is also a generating set for $I$.  Assume $\mathcal{F}$ is not a minimal generating  set.
Without loss of generality we may write $f_1=\sum_{i=2}^r g_i f_i$ with $g_i\in B$, for $i=2,\dots, n$.
Then
\begin{equation}
\label{eq:star}
f_1^*= \sum_{\genfrac{}{}{0pt}{}{i=2}{g_i f_i \notin \nn^3}}^r g_i^*f_i^*,
\end{equation}
contradicting the fact that $\mathcal{F}$ is a minimal generating set for $I^*$.

Conversely, assume $\mathcal{F}=\{f_1, \dots, f_r\}$ is a minimal generating system for $I$.
Since $I^*$ is generated in degree $2$, it suffices to show that  $f^*\in (f_1^*, \dots, f_r^*)$ for  $f\in I$ with  $\deg f^*=2$.
We may write $f=\sum_{i=1}^r g_i f_i$ with $g_i\in B$, $i=1, \dots, r$.
Then
$$
f^*= \sum_{\genfrac{}{}{0pt}{}{i=1}{g_i f_i \notin \nn^3}}^r g_i^* f_i^*,
$$
because $\deg f^*=2$.

Part (b)  follows from part (a) and the fact that $\height I = \height I^*$.
\end{proof}

There are further restrictions for the multiplicity of a quadratic semigroup $H$ if we assume that $\gr_\mm K[H]$ is Cohen-Macaulay.
Before proving them, we list in the next lemma some useful arithmetic properties of the generators of a  quadratic numerical semigroup.

\begin{Lemma}
\label{lemma:quad-semi}
Let $H$ be a numerical semigroup minimally generated by $a_1<a_2<\dots <a_n$ with  $n>1$.
If $H$ is quadratic, then
\begin{enumerate}
\item[{\em (a)}] there exist $ k, \ell \geq 2$ such that $a_1| a_k +a_\ell$.
\item [{\em (b)}] $2 a_i \in \langle a_1, \dots a_{i-1}, a_{i+1}, \dots a_n\rangle$, for all $2\leq i \leq n$.
\end{enumerate}
\end{Lemma}

\begin{proof}
We may pick $\mathcal{B}=\{ f_1, \dots, f_r \}$ a minimal standard basis  of $I_H$ consisting of binomials.
Since  $H$ is quadratic, $\deg f_j^*=2$ for $j=1, \dots, r$.

For all $i=1, \dots, n$, let $c_i$ be the smallest positive integer such that $c_i a_i$ is a sum of the other generators.
Then for any $i$ there exists $1\leq n_i \leq r$ such that  $f_{n_i}= x_i^{c_i}-\dots$.

Assume $f_{n_1}=x_1^{c_1}-\prod_{j\neq 1} x_j^{r_{j}} \in \mathcal{B}$, which gives the relation
\begin{eqnarray}
\label{eq:c1}
c_1 a_1= \sum_{j\neq 1} r_{j} a_j \text{ with  $r_{j}$ nonnegative integers}.
\end{eqnarray}
Since $a_1= e(H)$ we get $c_1> \sum_{j\neq 1} r_{j}$ and $f_{n_1}^*= \prod_{j\neq 1} x_j^{r_{j}}$.  As $\deg f_{n_1}^*=2$, using \eqref{eq:c1}
we conclude that there exist $k, \ell>1$ such that $c_1 a_1 = a_k+ a_\ell$.

Let $i>1$. Then $g_i=x_1^{a_i}-x_i^{a_1}$ is in $I_H$ and $g_i^*=x_i^{a_1} \in I_H^*$. Therefore, there exists a pure power of $x_i$, namely $x_i^2$,
among the terms of $f_1^*, \dots, f_r^*$. On the other hand, $x_i^{c_i}$ is the smallest pure power of $x_i$ occurring in any binomial in  $I_H$.
We get that  $1<c_i\leq 2$, hence $c_i=2$. This concludes the proof.
\end{proof}

\begin{Example}
\label{ex:one}{\em
A Singular (\cite{Sing}) computation, and also Propositions \ref{prop:3-semi-quad} and \ref{prop:arithmetic} show that the semigroups
$H_1= \langle 3,4,5\rangle$  and $H_2= \langle 4,5,6\rangle$ are quadratic. With notation as in Lemma \ref{lemma:quad-semi}
we notice that $a_1| a_2+ a_3$, respectively $a_1| 2a_3$.
Therefore the indices $k$ and $\ell$ in Lemma \ref{lemma:quad-semi}(a) may be distinct or the same. }
\end{Example}

\begin{Remark}{\em
Part (a) of Lemma \ref{lemma:quad-semi} appeared as Proposition 5.11 in the Ph.D. thesis of the second author \cite{St-thesis}.
There it was derived using the topological properties of the intervals in a quadratic semigroup, as described in \cite{RS}. }
\end{Remark}

\begin{Theorem}
\label{thm:forbidden}
Let $H$ be a quadratic numerical semigroup minimally generated by $a_1<\dots <a_n$ such that $\gr_\mm K[H]$ is Cohen-Macaulay.
The following hold:
\begin{enumerate}
\item[\em{(a)}] either
$
n \leq e(H) \leq 2^{n-1}-2^{n-3}, \text{ or } e(H)= 2^{n-1}.
$
\item[\em{(b)}] If $e(H)=2^{n-1}-2^{n-3}$, then   $I_H^*$ is an almost CI ideal.
\end{enumerate}
In the situation of {\em (b)},   $I_H^*$ has a quadratic Gr\"obner basis with respect to 
the degree reverse lexicographic order induced by $x_n>\dots >x_1$. 
\end{Theorem}

\begin{proof}
(a) Assume $e(H) <2^{n-1}$.   Let $S=K[x_1, \dots, x_n]$.
Since $S/I_H^*$ is Cohen-Macaulay, we get that $x_1$ is   regular  on $S/I_H^*$. Going modulo $x_1$ we have
$$
e(H)=e(S/I_H^*)= e(S/(x_1, I_H^*)) = e(K[x_2,\dots,x_n]/J),
$$
where $J$ denotes the image of the ideal $I_H^*$ through the $K$-algebra map sending $x_1$ to $0$ and keeping the other variables unchanged.

By Lemma \ref{lemma:quad-semi}, for $j=2, \dots, n$ there exist distinct polynomials $g_j=x_j^2-m_j$ in $J$,
where $m_j=0$ or  $m_j=x_{i} x_{k}$ with $i<j<k$. Therefore, with respect to the degree reverse lexicographic order  induced by $x_1<x_2<\dots$ we have that
 $\ini_<(g_j)=x_j^2$, for $2\leq j \leq n$.

By Theorem \ref{thm:bounds} we get that $I_H^*$ is not a CI, hence $\mu(I_H^*)=\mu(J) \geq n$. So besides $g_2, \dots, g_n$
there is at least one more generator $f$ in $J$, $\deg f=2$, and without loss of generality we may assume that $f$ is either a monomial
or a homogeneous binomial whose terms are not pure powers.
We let $T= K[x_2, \dots, x_n]/(x_2^2, \dots, x_n^2)$ and $g$ be the residue class of $\ini_<(f)$ in $T$.
Hence
\begin{eqnarray*}
e(H) &=& \ell(K[x_2,\dots,x_n]/J ) = \ell(K[x_2,\dots,x_n]/ \ini_< (J)) \\
		 &\leq& \ell(K[x_2, \dots, x_n]/(x_2^2, \dots, x_n^2, \ini_<(f))) \\
		 &=& \ell (T/(g)) =\ell ((0: _T g)).
\end{eqnarray*}

Let us denote $x_U=\prod_{k\in U} x_k$, for all $U\subset [2, n]$, where $x_\emptyset=1$.

If $g=x_ix_j$, with $i\neq j$,  a $K$-basis for $(0:_T g)$ is given by the monomials
$$
\{ x_i x_U: U\subset [2, n]\setminus \{i,j\}  \}
\cup \{ x_j x_V: V\subset[2,n]\setminus \{i,j\} \}
\cup \{x_i x_j x_W: W\subset[2,n]\setminus \{i,j\}	\},
$$
hence $e(H)\leq \dim_K (0:_T g)= 3 \cdot 2^{n-3}= 2^{n-1}-2^{n-3}$.

(b) From the above arguments we note that the equality $e(H)= 2^{n-1}-2^{n-3}$ holds if and only if
 $\ini_<(J)= (x_2^2, \dots, x_n^2, \ini_<(f))$, i.e.
$\{g_2, \dots,g_n, f \}$ is a (clearly reduced) Gr\"obner basis of $J$. Therefore, $\mu(J)= n$, which reads as  $I^*_H$ being
an almost CI ideal.

Clearly $g_1, \dots, g_{n-1}$ and $f$ may be lifted to $S$ to quadratic polynomials   $f_1, \dots, f_n$ in $I^*_H$, respectively,  such that
$\ini_<(f_i)= \ini_<(g_i)$ for $1\leq i \leq n-1$ and  $\ini_<(f_n)= \ini_<(f)$. Let $\mathcal{F}= \{ f_1, \dots, f_n\}$.

We claim that $\mathcal{F}$ is a Gr\"obner basis for $I_H^*$.
Clearly $ (\ini_<(f_1), \dots, \ini_<(f_n))  \subseteq \ini_<(I^*_H)$. For the reverse inclusion it is enough to show that these
two ideals have the same Hilbert series.

Indeed, since $x_1$ is regular on $S/I^*_H$ and on $S/(\ini_<(f_1), \dots , \ini_<(f_n))$
we may write
\begin{eqnarray*}
H_{S/\ini_<(I^*_H)}(t)= H_{S/I^*_H}(t)= \frac{1}{1-t} H_{S/(x_1, I^*_H)}(t)=  \frac{1}{1-t} H_{K[x_2, \dots, x_n]/J}(t),  \\
H_{S/(\ini_<(f_1), \dots, \ini_<(f_n))}(t)= \frac{1}{1-t} H_{K[x_2, \dots, x_n]/\ini_<(J)} (t)  =\frac{1}{1-t} H_{K[x_2, \dots, x_n]/J}(t).
\end{eqnarray*}

This ends the proof.
\end{proof}

In Example \ref{ex:aci} we present a family of semigroups satisfying the hypotheses of Theorem \ref{thm:forbidden}(b).

\begin{Remark}
\label{rem:aci-nonkoszul}
{\em
Note that the converse to the implication in Theorem \ref{thm:forbidden}(b) is not true. One may check with Singular (\cite{Sing}) that $H= \langle  11,13,14,15,19 \rangle$ is quadratic,
almost complete intersection, but it is not a  Koszul semigroup.}
\end{Remark}

It is natural to ask the following:

\begin{Question} {\em
\label{que:forbidden-aci}
Do the conclusions of Theorem \ref{thm:forbidden} stay true for  any quadratic numerical semigroup $H$, without assuming that $\gr_\mm K[H]$ is Cohen-Macaulay? }
\end{Question}

The answer is positive if $\embdim(H) \leq 5$, as shown by the second author in \cite[Proposition 1.5, Theorem 1.8]{St-quadcm}. 
Moreover, if $\embdim(H)\leq 7$ it follows from 
Rossi and Valla's work (see \cite[Theorem 5.9]{RV}) that if $H$ is quadratic and not a complete intersection, then $e(H) \leq 2^{n-1}-2^{n-3}$.

\begin{Proposition}
\label{prop:g-quadratic}
Let $H$ be a quadratic semigroup with $\embdim(H)=n$. Assume that  $e(H)$ attains one of the extremal values, namely $e(H) \in \{n, 2^{n-1}\}$,
or that $\gr_\mm K[H]$ is Cohen-Macaulay and $e(H)= 2^{n-1}-2^{n-3}$. Then $H$ is $G$-quadratic and in particular it is a Koszul semigroup.
\end{Proposition}

\begin{proof}
In case $e(H)=n$, as noted in the proof of Proposition \ref{prop:linear} we have that $R=\gr_\mm K[H]$ is Cohen-Macaulay,  $x_1$ is regular on $R$ and
$(x_1,I_H^*)=(x_1, (x_2, \dots, x_n)^2)$, which is a monomial ideal. Therefore $R/(x_1)$ is G-quadratic, and by using Conca's \cite[Lemma 4.(2)]{Conca-quadrics},
we conclude that $\gr_\mm K[H]$ is G-quadratic, too.

In case $e(H)=2^{n-1}$, by Proposition \ref{prop:ci-local} we have that $I_H^*$ is a CI. On the other hand,
by Lemma \ref{lemma:quad-semi}, for $j=2, \dots, n$ there exist distinct $n-1$ quadratic  polynomials $f_j=x_j^2-m_j$ in $I_H^*$,
where $m_j=0$ or  $m_j=x_{i} x_{k}$ with $i<j<k$. Therefore $I_H^*=(f_2, \dots, f_n)$.  
With respect to the degree reverse lexicographic order induced by $x_1>x_2>\dots$ we have that
$\gcd(\ini_<(f_i), \ini_<(f_j))= \gcd(x_i^2, x_j^2)=1$ for all $2\leq i<j \leq n$, hence $\{ f_2, \dots, f_n\}$ is a quadratic Gr\"obner basis for $I^*_H$.

The case $e(H)=  2^{n-1}-2^{n-3}$ was discussed in Theorem \ref{thm:forbidden}(b).
\end{proof}

\begin{Remark}{\em
It was proven in \cite[Theorem 5.2]{HRW} that if $\Lambda$ is any affine semigroup such that $K[\Lambda]$ is Cohen-Macaulay and of minimal multiplicity, 
then $\gr_\mm K[\Lambda]$ is Koszul. With essentially the same argument as in the proof of Proposition \ref{prop:g-quadratic} one obtains that 
$\gr_\mm K[\Lambda]$ is $G$-quadratic.}
\end{Remark} 

\begin{Remark}
\label{rem:cdnr}
{\em
The fact from Proposition \ref{prop:g-quadratic} that if $e(H)=n$  then $\gr_\mm K[H]$ is Koszul, is  folklore, 
and it is also mentioned in the survey 
\cite{CdNR}. 

By applying \cite[\S 6, Proposition 8]{CdNR} it follows that if $e(H)=n+1$ and the Cohen-Macaulay type $\tau$ of $K[H]$ satisfies
$\tau < n-1$, then $\gr_\mm K[H]$ is Koszul.  
The proof can be easily continued to conclude that  $H$ is $G$-quadratic. 
}
\end{Remark}

\begin{Remark} {\em  Let $H$ be a numerical semigroup with $e(H)= 2^{\embdim(H)-1}$.
It is easy to see   that if   $I_H^*$ is CI,  then $I_H^*$ is generated in degree $2$.

However, if we assume that $I_H$ is a CI, can we also derive that $I_H^*$ is CI, hence  quadratic?}
\end{Remark}

\medskip


\section{Quadratic semigroups and   gluings}
\label{sec:gluing}

The following construction on numerical semigroups was introduced by Watanabe (\cite{Wat}) and extended later
by Delorme (\cite{Delorme}) and Rosales (\cite{Rosales}) who seems to have coined the name {\em gluing}.

\begin{Definition}{\em
Given the numerical semigroups $H_1$ and $H_2$ and  the   integers $c_1, c_2 >1$,
the semigroup $H= \langle c_1H_1, c_2H_2 \rangle$ is called a {\em gluing of $H_1$ and $H_2$}
if $c_1\in H_2$, $c_2\in H_1$ and $\gcd(c_1, c_2)=1$.}
\end{Definition}

We are interested in the situation when one of the glued semigroups is $\NN$ itself.
\begin{Definition}
\label{def:gluing}
{\em
Given the numerical semigroup $L$, the   integers $c >1$ and $\ell$ such that $\ell \in  L\setminus G(L)$ and $\gcd(c, \ell)=1$,
the numerical semigroup
$H= \langle cL, \ell \NN\rangle$ is called a {\em simple gluing} of $L$.

If moreover $c=2$ we call $\langle 2L, \ell \rangle$ a {\em quadratic gluing}.}
\end{Definition}

For the rest of the paper, when we describe a semigroup as $\langle cL, \ell \rangle$ we assume it is obtained from a simple gluing as in Definition \ref{def:gluing}.

If in Definition \ref{def:gluing} we allowed $\ell \in G(L)$, then (exactly) one of the generators of $H$ is superfluous
and $\embdim(H)=\embdim(L)$, cf. \cite[Proposition 10.(ii)]{Delorme}.
We want to avoid this case in order to have $\embdim(H)= \embdim(L)+1$.

Consider the simple gluing $H=\langle cL, \ell \rangle$.  Assume $\embdim(L)=n-1$. We may write
\begin{equation}
\label{eq:glue-eq}
\ell =\sum_{i=1}^{n-1} \lambda_i a_i
\end{equation}
 as a sum  of the minimal generators
$a_1,\dots, a_{n-1}$ for $L$ and such that $\sum_{i=1}^{n-1} \lambda_i \geq 2$ is maximal.
This gives a so called {\em gluing relation}
\begin{equation}
\label{eq:glue-poly}
f=x_{n}^c- \prod_{i=1}^{n-1}x_i^{\lambda_i}  \in I_H.
\end{equation}
The largest value of $\sum_{i=1}^{n-1} \lambda_i$ such that \eqref{eq:glue-eq} holds is called the {\em order} of $\ell$ in $L$ and it is denoted $\ord_L(\ell)$.

By convention, unless it is otherwise specified,
when we work with the toric ideal of $H$ we assume that $\ell$ corresponds to the last variable $x_n$ and the rest of the generators of $H$ are ordered as in $L$.
If we denote $S=K[x_1,\dots, x_n]$, it is easy to see
(e.g. in the proof of Lemma 1 in \cite{Wat}) that
\begin{equation}
\label{eq:glue}
I_H= (I_L S , f).
\end{equation}

The next result describes the transfer of quadraticity (and of the Koszul property) via gluings.

\begin{Theorem}
\label{thm:glue-quadratic}
Consider the numerical semigroup $H= \langle cL, \ell \rangle$ where $\gcd(c, \ell)=1$, $c>1$ and $\ell \in L\setminus G(L)$. Let $f$ be a gluing relation as in \eqref{eq:glue-poly}.

If $c \leq \ord_L(\ell)$ the following hold:
\begin{enumerate}
\item [{\em (a)}] $I_H^*= (I_L^* S , f^*)$;
\item [{\em (b)}] $H$ is quadratic \iff  $c=2$ and $L$ is quadratic;
\item [{\em (c)}] $H$ is Koszul \iff $c=2$ and $L$ is Koszul;
\item [{\em (d)}] $I_H^*$ has a quadratic Gr\"obner basis \iff $c=2$ and $I_L^*$ has a quadratic Gr\"obner basis.
\end{enumerate}
\end{Theorem}

With notation as above, the condition $c\leq \ord_L(\ell)$ implies, using \eqref{eq:glue-eq}, that $c \cdot e(L) \leq \ord_L(\ell) \cdot e(L) \leq \ell$.
Since $\gcd(c, \ell)=1$ and $c>1$  we obtain
\begin{equation}
\label{eq:small-c}
e(\langle cL, \ell \rangle)= c\cdot e(L) < \ell.
\end{equation}

For the proof of Theorem \ref{thm:glue-quadratic} and later we need the following technical lemmas.

The first one follows from \cite[Lemma, part (a), pp. 185]{He-reg}.
\begin{Lemma}
\label{lemma:i-star}
Let $I$ be an ideal in the polynomial ring $S=K[x_1,\dots, x_n]$  such that $I \subset \mm= (x_1, \dots, x_n)$.
 If $f\in S$ is such that $\deg f^* >0$ and 
  $f^*$ is regular on $S/I^*$, then
$$
(I, f)^*= (I^*, f^*).
$$
\end{Lemma}

As an immediate corollary of Lemma \ref{lemma:i-star} we obtain the next statement.
\begin{Lemma}
\label{lemma:strict-ci}
Let $f_1, \dots, f_r$ be a regular sequence in $S=K[x_1, \dots, x_n]$ such that $f_1^*, \dots, f_r^*$ is a regular sequence, too.
Then
$$
(f_{i_1}, \dots, f_{i_s})^*= (f_{i_1}^*, \dots, f_{i_s}^*),
$$
for   $s=1,\dots, r$ and   $1 \leq i_1 < \dots < i_s \leq r$.
\end{Lemma}

\begin{Lemma}
\label{lemma:glue-istar}
Consider the ideal $I \subset K[x_1, \dots, x_{n-1}] \subset S=K[x_1, \dots, x_n]$ and the polynomial $f= x_n^c-m$ with $c>0$ and $m\in K[x_1, \dots, x_{n-1}]$. Then
\begin{enumerate}
\item[{\em (a)}] $f$ is regular  on $S/IS$, and
\item[{\em (b)}] if $\deg m \geq c$, then $f^*$ is regular on $S/I^*S$ and $(IS, f)^*= (I^*S, f^*)$.
\end{enumerate}
\end{Lemma}

\begin{proof}
Let $<$  be the lexicographic term order on $S$ induced by $x_n>x_{n-1}>\dots >x_{1}$. Then $\ini_<(f)= x_n^c$ and under the extra  requirement at (b) we also have
$\ini_<(f^*)=x_n^c$. Since the variable $x_n$ does not appear in the monomial generators for $\ini_< (IS)$ and $\ini_<(IS^*)$ we
get that $x_n^c$ is regular on $S/\ini_<(I)$ and on $S/\ini_<(IS^*)$.  By \cite[Proposition 15.15]{Eis} we obtain part (a) and the first half of (b).

The second part of (b) results from Lemma \ref{lemma:i-star}.
\end{proof}

We may now go back to the proof of Theorem \ref{thm:glue-quadratic}.
\begin{proof} (of Theorem \ref{thm:glue-quadratic})
We apply Lemma \ref{lemma:glue-istar} to $I_L$ and the gluing relation $\eqref{eq:glue-poly}$ to conclude that (a) holds.

For (b), if $\mathcal{G}=\{f_1, \dots, f_r\}$ is any minimal standard basis for $I_L$, by (a) we get that
$\mathcal{G}'=\mathcal{G} \cup \{f\}$ is  a standard basis for $I_H$.

We claim that $\mathcal{G}'$ is minimal. Indeed,
since
the variable $x_n$ appears in $\mathcal{G}'$ only in $f^*$, we can not remove $f$ from $\mathcal{G}'$.
If we could remove some $f_j$, say $f_r$, then $f_r^*= \sum_{i=1}^{r-1} h_i f_i^* + h f^*$ for $h, h_i \in S$, $i=1,\dots, r$.
By Lemma \ref{lemma:glue-istar},  $f^*$ is regular on $S/I^*S$ and we get $h \in I^*S$, which contradicts the minimality of $\mathcal{G}$.

Therefore $I_H^*$ is generated in degree $2$ if and only if $I^*_L$ is generated in degree $2$ and $\deg f^*=2$. This gives (b).

For part (c) we remark that by (b) for any of the two implications that need to be checked we have $\deg f^*=c=2$.
 According to \cite[Lemma 2]{BF-poincare},
since $f^*$ is regular  on $S/I^*S$ and of degree $2$, the ring  $S/I^*_LS$ is Koszul if and only if
$$\frac{S/I^*_LS} {(f^*) S/I^*_LS} \cong S/(I^*_LS, f^*) = S/I^*_H \cong \gr_\mm K[H]$$
is Koszul.
Clearly $S/I^*_LS$ is Koszul if and only if $K[x_1, \dots, x_{n-1}]/I^*_L$ (which is isomorphic to $\gr_\mm K[L]$) is Koszul, hence (c) holds.

For part (d) we first assume $I_H^*$ has a quadratic Gr\"obner basis $\mathcal{G}$  with respect to some term order $<$.
Without loss of generality we may assume $\mathcal{G}$ is reduced, and because $I_H^*$ is generated by binomials and/or monomials,
then it is well-known that $\mathcal{G}$ consists of monomials and/or binomials of degree $2$.

Note that the variable $x_n$ may not appear with exponent different from $2$ in any monomial term of any polynomial in $\mathcal{G}$.
For an exponent $3$ or larger, that is clear by the quadraticity of $\mathcal{G}$.
Also, if we assume that there  exists a relation $x_n x_k- \prod_{i=1}^{n-1} x_i^{\mu_i} \in I_H$ with $k<n$, and $\sum_{i=1}^{n-1}\mu_i \geq 2$,
we get $\ell+ c a_k = c\sum_{i=1}^{n-1} \mu_i a_i$, which is false since $\ell$ and $c$ are coprime.

By \eqref{eq:small-c}, arguing as in the proof of Lemma \ref{lemma:quad-semi} we obtain $x_n^{ca_1} \in I_H^*$, and also
$x_n^{ca_1} \in \ini_<(I_H^*)$. Therefore there exists $g=x_n^2-m$ in $\mathcal{G}$ such that $\ini_<(g)=x_n^2$ and $m$ is a monomial or $0$.
If $g_1=x_n^2-m_1$ were another element in $\mathcal{G}$ containing $x_n^2$ we may reduce it further with $g$, and this contradicts the fact that
$\mathcal{G}$ is a reduced Gr\"obner basis.
Consequently, the variable $x_n$ does not divide any monomial term of any element of $\mathcal{G}'=\mathcal{G}\setminus \{g\}$.

Let $J= (\mathcal{G}')$. Clearly $J \subset I_L^*S$. It is easy to see, by the Buchberger criterion,
 that $\mathcal{G'}$ is a reduced Gr\"obner basis for $J$.  Hence
$$\
\ini_<(I_H^*)= \ini_<(J)+ (x_n^2).
$$
Since  $f^*$ is regular on $S/I^*_LS$ and   $x_n^2$  is regular on $S/\ini(J)$, using part (a) we obtain that
\begin{eqnarray*}
\Hilb_{S/I_H^*}(t)  &=& (1-t^2) \Hilb_{S/I_L^*S}(t), \\
\Hilb_{S/\ini_<(I_H^*)}(t) &=& (1-t^2) \Hilb_{S/\ini_<(J)}(t).
\end{eqnarray*}
By Macaulay's Theorem (\cite[Theorem 2.6]{EH}) we have $\Hilb_{S/I_H^*}(t)= \Hilb_{S/\ini_<(I_H^*)}(t)$ and
$\Hilb_{S/J}(t)=\Hilb_{S/\ini_<(J)}(t)$. Hence $ S/I_L^*S$ and $S/J$ have the same Hilbert series, which together with $J\subseteq I_L^*S$ gives
$J= I_L^*S$.

Consequently $\mathcal{G}'$ is the reduced   Gr\"obner basis for $I_L^*S$, and also for $I_L^*$, and we are done.
Indeed, if $q \in I_L^*S$, then $\ini_<(q)$ is not divisible by $x_n^2$, and  hence it is divisible by $\ini_<(g')$ for some $g' \in \mathcal{G}'$.

For the converse, assume $c=2$ and that $I_L^*$ has a quadratic Gr\"obner basis $\mathcal{G}'$ with respect to
 some term order $<'$ on $K[x_1, \dots, x_{n-1}]$.
Let $<$ be the block order $(lex, <')$ where we first apply the lexicographic term order on the variable $x_n$ and
for ties we apply $<'$ on the rest of the monomial in the variables $x_1, \dots, x_{n-1}$.
Then $\ini_<(f^*)= x_n^2$.

We claim that $\mathcal{G}= \mathcal{G}' \cup \{f^*\}$ is a Gr\"obner basis for $I_H^*$. Indeed,  $\mathcal{G}'$ is already a Gr\"obner basis, hence
the only $S$-pairs to be checked involve $f^*$ and $g \in \mathcal{G'}$. Since their leading terms are coprime, $S(f^*, g)\stackrel{\mathcal{G}}{\rightarrow} 0$.
This finishes the proof.
\end{proof}

Since the hypothesis of Theorem \ref{thm:glue-quadratic} implies that $\ord_L(\ell) \geq 2$, we obtain the following corollary.
\begin{Corollary}
\label{cor:glue-2}
Let $L$ be any numerical semigroup and $\ell\in L \setminus G(L)$ an odd integer.
Then the semigroup $ \langle 2L, \ell \rangle$ is quadratic (Koszul) if and only if $L$ is quadratic (Koszul).
\end{Corollary}

We will also need the following consequence.
\begin{Corollary}
\label{cor:ci-aci}
Let $L$ be a quadratic numerical semigroup and $\ell\in L \setminus G(L)$ an odd integer. Denote $H= \langle 2L, \ell \rangle$.
The following hold:
\begin{enumerate}
\item[{\em (a)}]   $ H$ is a complete intersection  semigroup \iff $ L$ is a complete intersection semigroup;
\item[{\em (b)}] if $I^*_L$ is an almost CI, then   $I_H^*$ is an almost CI, too.
\end{enumerate}
\end{Corollary}

\begin{proof}
$H$ is quadratic by Corollary \ref{cor:glue-2}. Since $e(H)=2 e(L)$, part (a) follows from Theorem \ref{thm:bounds}(c).
For part (b), denote $\embdim(L)=n-1$. Hence $\mu(I_L^*)= n-1$, and by Theorem \ref{thm:glue-quadratic}(a) we obtain $\mu(I_H^*) \leq n$.
If $\mu(I_H^*)=n-1$, by part (a) we get that $L$ is CI, which is false.
Therefore $\mu(I_H^*) =n$.
\end{proof}

\begin{Remark} {\em In general, if $\gr_\mm K[L]$ is a CI, it it not always true that for $H=\langle c L, \ell \rangle$ its tangent cone $\gr_\mm K[H]$ is CI, too.
If we let $H=\langle 4 \langle 2, 5 \rangle, 7 \rangle =\langle 7, 8,20 \rangle$, mapping the variables to the generators taken in increasing order,
 we have $I_H^*=(x_3^2, x_2 x_3, x_1^4 x_3, x_2^7)$, which is not even a Cohen-Macaulay ideal.}
\end{Remark}

\begin{Remark} {\em
If $c > \ord(\ell)$ we have less control over the output of the gluing, even if $\deg f^*= \ord_L(\ell)=2$.

Let $L=\langle 4,6,7,9 \rangle$. Clearly $\ord_L(8)=\ord_L(10)=2$. It is easy to compute (e.g. using Singular \cite{Sing} or CoCoa \cite{Cocoa})
$$
I_L^*= (x_2^2, x_2 x_3-x_1 x_4, x_3^2, x_2 x_4, x_3 x_4, x_4^2)
$$
and check that the listed generators are a quadratic Gr\"obner basis with respect to the degree reverse lexicographic order on $x_1>x_2>x_3>x_4$.

We may also check  that for the gluing
$H_1= \langle 3L, 8 \rangle = \langle 12, 18, 21, 27, 8 \rangle$,
the ideal
$$
I_{H_1}^*= (x_1^2, x_2^2, x_2x_3-x_1x_4, x_3^2, x_2x_4, x_3x_4, x_4^2 )
$$
has a quadratic Gr\"obner basis with respect to the same term order as above, however for the gluing
$H_2= \langle 3L, 10 \rangle = \langle 12, 18, 21, 27, 10 \rangle$ the ideal
$$
I_{H_2}^*= ( x_1 x_2, x_2^2,  x_2 x_3-x_1 x_4, x_3^2, x_2x_4, x_3x_4, x_4^2, x_1^3x_3-x_4x_5^3, x_1^4-x_2x_5^3)
$$
is not generated in degree $2$.

}
\end{Remark}

\begin{Remark} {\em Arbitrary gluings of quadratic numerical semigroups may not be quadratic anymore. If we consider the Koszul semigroups
 $H_1= \langle 2,3 \rangle$, $H_2= \langle 2,5 \rangle$ and the gluing  $H= \langle 7H_1, 5 H_2 \rangle= \langle  14, 21, 10, 15\rangle $, we have that
$e(H)> 8$, and according to our Theorem \ref{thm:bounds}, $H$ is not quadratic. Note that by Delorme's Theorem~\ref{thm:Delorme}, $H$ is a complete intersection semigroup.
}
\end{Remark}

\begin{Example}
\label{ex:watanabe} {\em In \cite[Lemma 3]{Wat} K. Watanabe  shows that for any odd integer $a>0$  the semigroup
$$
W_n(a)= \langle 2^n, 2^n+a, 2^n+2a, 2^n+4a, \dots, 2^n+2^i a, \dots, 2^n+ 2^{n-1}a \rangle
$$
 is a complete intersection of $\embdim(W_n(a))=n+1$.
We   prove that it is a  Koszul semigroup.

It is easy to see by induction  on $n$ that $W_n(a)$ may be otained by simple gluings
by the rule  $W_{n}(a)= \langle 2 W_{n-1}(a), 2^n+a \rangle$, for any $n>1$ starting from $W_1(a)= \langle 2, 2+a \rangle $.
Clearly $W_1(a)$ is Koszul, hence by induction using Theorem \ref{thm:glue-quadratic}(c) we get that $W_n(a)$ is a Koszul semigroup for any $n$.
}
\end{Example}

Our next result shows that the quadratic (and Koszul) semigroups $H$ for which $K[H]$ is a complete intersection
are obtained from $\NN$ by a sequence of quadratic gluings.
We first recall Delorme's structure theorem for complete intersection semigroup rings.

\begin{Theorem} (Delorme, \cite[Proposition 9]{Delorme})
\label{thm:Delorme}
Let $H$ be a numerical semigroup. Then $K[H]$ is a complete intersection if and only if either $H=\NN$ or there exist  $H_1$, $H_2$
numerical semigroups, coprime integers $c_1, c_2$ such that  $H= \langle c_1 H_1, c_2 H_2 \rangle $,
$c_1\in H_2\setminus G(H_2)$, $c_2\in H_1\setminus G(H_1)$, and $K[H_1]$, $K[H_2]$ are complete intersections.
\end{Theorem}

\begin{Theorem}
\label{thm:quad-ci}
Let $H$ be a numerical semigroup such that $K[H]$ is a complete intersection.
The following are equivalent:
\begin{enumerate}
\item[{\em (a)}] $H$  is obtained uniquely from $\NN$ by a series of quadratic  gluings

$H_0= \NN$, $H_1=\langle 2H_0, \ell_1 \rangle, \dots,  H_{n-1}=\langle 2H_{n-2}, \ell_{n-1} \rangle=H$,

where $\ell_i$ is an odd integer in $H_{i-1}\setminus G(H_{i-1})$ for $i=1, \dots, n-1$,
\item[{\em (b)}] $H$ is Koszul,
\item[{\em (c)}] $H$ is quadratic.
\end{enumerate}
\end{Theorem}

\begin{proof}
For (a)$ \Rightarrow$ (b) we start with $H_0$ which is Koszul and we repeateadly use
Theorem \ref{thm:glue-quadratic} to derive that $H_1, \dots, H_{n-1}=H$ are Koszul, as well.
The implication (b)  $\Rightarrow$ (c) is clear.

For (c) $\Rightarrow$ (a) we assume that $H$ is quadratic. We prove the existence  of a chain of
quadratic gluings by induction on $n= \embdim (H)$.
For $n=1$ we have $H=\NN$ and there is nothing to prove. If $n=2$, then $H=\langle a, b \rangle$ with $a < b$ coprime.
This gives $I_H=(x_1^b-x_2^a)$ and $I_H^*=(x_2^a)$. Since $H$ is quadratic we get $a=2$, and $b$ odd. Hence $H= \langle 2\NN, b \rangle$ is a simple gluing  as desired.

Assume that all CI  quadratic semigroups of embedding dimension smaller than $n$ may be obtained as in (a).
Let $H$ be a quadratic CI semigroup with $\embdim(H)=n$.
By Delorme's Theorem \ref{thm:Delorme},  $H= \langle c_1 U, c_2 V \rangle$ where $U$ and $V$ are CI and
 with $\embdim(U)=r, \embdim(V)=n-r$.

If $r=1$, by Theorem \ref{thm:glue-quadratic}(b) $c_2=2$ and  $V$ must be quadratic, i.e. $H= \langle c_1, 2 V \rangle$.
By the induction hypothesis, we may obtain $V$ from $\NN$ via quadratic gluings, and we are done.
The case $n-r=1$ is treated similarly.

Assume $r>1$ and $n-r >1$. If $I_{U}= (f_1,\dots, f_{r-1})$ and $I_{V} =(f_{r}, \dots, f_{n-2})$,
then from the proof of Delorme's Theorem \ref{thm:Delorme}, $I_H=I_{U}+I_{V}+ (f_{n-1})$, where the gluing relation
$f_{n-1}$ is obtained in a similar way as in \eqref{eq:glue-poly}.
Since $H$ is quadratic, by Lemma \ref{lemma:quad-basis} we get that $f_1, \dots, f_{n-1}$ is a minimal standard basis of $I_H$,
and since $\height I_H=\height I_H^* =n-1$ we get that $f_1, \dots, f_{n-1}$ and  $f_1^*, \dots, f_{n-1}^*$ are regular sequences.
By Lemma \ref{lemma:strict-ci} we obtain that $f_1, \dots, f_{r-1}$, respectively $f_r, \dots, f_{n-2}$  form a standard basis for $I_{U}$ and $I_{V}$, respectively.
This gives that $U$  and $V$ are quadratic CI semigroups.

By Theorem \ref{thm:bounds}, $e(U)=2^{r-1}$ and $e(V)=2^{n-r-1}$. Without loss of generality we may assume $e(H)=c_1 e(U)$.
Then $c_1= 2^{n-r }$ and $c_2 \in U$ is odd.
By the induction hypothesis we may obtain $V$ from a quadratic gluing: $V= \langle 2W, \ell \rangle $, where $W$ is quadratic and CI, and $\ell\in W\setminus G(W)$ is odd.
Hence $e(W)= 2^{n-r-2}$.
By Delorme's Theorem \ref{thm:Delorme} the numerical semigroup $Z=\langle 2^{n-r-1} U, c_2 W \rangle$ is CI, obtained by gluing the CI semigroups $U$ and $W$.
Note that we may write
$$
H =\langle c_1U, c_2 V \rangle= \langle 2^{n-r}U, c_2\langle 2W, \ell \rangle \rangle =
\langle 2^{n-r}U, 2c_2W, c_2 \cdot \ell \rangle=  \langle 2Z, c_2 \cdot \ell \rangle.
$$
Since $c_2 \cdot \ell$ is odd and $c_2\cdot \ell \in Z\setminus{G(Z)}$ (because $\ell \in W\setminus G(W)$) we may apply
Corollary \ref{cor:glue-2} to obtain that $Z$ is a quadratic numerical semigroup. Since $Z$ is  CI and $\embdim(Z)=n-2$,
by the   induction hypothesis it  may be obtained from $\NN$ by quadratic gluings, so the same is true for $H$.

The uniqueness of the decomposition follows from the fact that there is exactly one odd minimal generator for $H$, 
hence it must be chosen as $\ell_{n-1}$. 
\end{proof}

As an application of the gluing construction we present  an infinite family of quadratic almost-CI semigroups satisfying the upper bound in Theorem \ref{thm:forbidden}(b).
\begin{Example}{\em
\label{ex:aci}
Let $H_4= \langle 6, 7, 8,9 \rangle$. We may read the defining equations of $\gr_\mm K[H_4]$ from the proof of Proposition \ref{prop:arithm-gbasis} and
we have that $H_4$ is Koszul and $\mu(I_{H_4}^*)=4$, hence $I_{H_4}^*$ is an almost CI ideal.
We construct recursively the semigroups
$$
H_{n+4}= \langle 2 H_{n+3}, 3^{n+2} \rangle, \text{ for all } n>0.
$$
It is  easy to check by induction and using Theorem \ref{thm:glue-quadratic} that for all $n>0$:
\begin{enumerate}
\item[{(a)}] $3^{n+2} \in H_{n+3}$, hence $H_{n+4}$ is obtained by a simple gluing from $H_{n+3}$ and $\embdim(H_{n+4})=n+4$;
\item[{(b)}] $
H_{n+4}= \langle 2^n\cdot 6, 2^n\cdot 7, 2^n\cdot 8,2^n\cdot 9,\ 2^{n-1}\cdot 3^3,2^{n-2}\cdot 3^4, \dots, 2\cdot 3^{n+1}, 3^{n+2} \rangle
$, hence $e(H_{n+4})= 2^{n+3}-2^{n+1}$;
\item [{(c)}] $I_{H_{n+4}}= (x_2^2-x_1 x_3, x_3^2-x_2 x_4, x_2 x_3-x_1 x_4, x_1^3-x_4^2)+(x_i^3-x_{i+1}^2: 4 \leq i \leq n+3)$;
\item [{(d)}] $I_{H_{n+4}}^*=(x_2^2-x_1 x_3, x_3^2-x_2 x_4, x_2 x_3-x_1 x_4)+ (x_i^2: 4\leq i\leq n+4)$, hence it is an almost CI ideal.
\end{enumerate} }
\end{Example}

More general than Question \ref{que:forbidden-aci} we may ask:
\begin{Question}
For a given $n>1$ what is the possible multiplicity of any quadratic (or Koszul) semigroup $H$ with $\embdim(H)=n$?
\end{Question}

The results presented so far and in the next section show that there are examples  of Koszul semigroups whenever $n\leq e(H) \leq 2n-2$, $e(H)=2^{n-1}-2^{n-3}$ or $e(H)=2^{n-1}$.
We remark that the gluing construction described in Corollary \ref{cor:glue-2}  allows us to construct new quadratic (or Koszul)
semigroups with double multiplicity and of embedding dimension increased by $1$.

\medskip


\section{Quadraticity in some families of semigroups}
\label{sec:examples}

Let $H$ be a quadratic numerical semigroup of embedding dimension $n$. By the results described so far we have that
\begin{eqnarray}
\label{eq:bounds}
n \leq e(H)\leq 2^{n-1}.
\end{eqnarray}

These bounds are tight. Indeed, if $e(H)=n$, then $H$ is quadratic by Proposition~\ref{prop:linear}, and even $G$-quadratic, by Proposition \ref{prop:g-quadratic}.
The upper bound is reached, e.g. in Example \ref{ex:watanabe}.

In this section we study the quadratic property in some  families of numerical semigroups that have been considered in the literature.

\subsection{Koszul arithmetic and geometric sequences}

A sequence $a_1<a_2<\dots <a_n$ of nonnegative integers  is called an {\em arithmetic sequence}, respectively a {\em geometric sequence}, if
 there exists a $d$ such that $a_{i+1}=d+a_i$, respectively $a_{i+1}=da_i $ for $i=1, \dots, n-1$.

We show that the multiplicity  of quadratic semigroups generated by an arithmetic sequence is in the  lower part of the interval in \eqref{eq:bounds},
while for geometric sequences the multiplicity is maximal.

The next statement about the tangent cone of a numerical semigroup generated by an arithmetic sequence is of interest by itself.
We could not locate this result in the literature, so for the convenience of the reader we give a proof including the references on which our proof is based.

\begin{Proposition}
\label{prop:arithm-gbasis}
If the numerical semigroup $H$ is generated by an arithmetic sequence $a_1< \dots < a_n$, then $I_H^*$ is minimally generated by its reduced Gr\"obner basis with respect to the degree reverse lexicographic order induced by $x_1 >x_2>\dots >x_n$.
\end{Proposition}

\begin{proof}
Let $S=K[x_1,\dots, x_n]$. D.P. Patil  proved in \cite{Patil} that under our hypothesis on $H$,
the generators of the toric ideal $I_H$
depend on the unique positive integers $a$ and $b$ with $1\leq b\leq n-1$ such that $a_1= a(n-1)+b$. Namely,
$$
I_H= (x_i x_{j+1} -x_{i+1}x_j : 1\leq i < j \leq n-1)+ (x_n^a x_{b+i}-x_1^{a+d}x_i: 1\leq i \leq n-b).
$$
It is shown in the proof of  \cite[Proposition 2.5]{HeS} that these generators of $I_H$ are also a standard basis of $I^*_H$ (alternatively see \cite[Corollary 2.4.(iii)]{Sha-Za-1}), hence
$$
I^*_H= (x_i x_{j+1} -x_{i+1}x_j : 1\leq i < j \leq n-1)+ (x_n^a x_{b+i} : 1\leq i \leq n-b).
$$
Denote $f_{ij}=x_i x_{j+1} -x_{i+1}x_j$ for $ 1\leq i < j \leq n-1$ and $g_i=x_n^a x_{b+i}$ for $1\leq i \leq n-b$.

We verify  Buchberger's criterion (see \cite[Theorem 2.14]{EH}) for
$$
\mathcal{G}= \{ f_{ij}:1\leq i < j \leq n-1\} \cup \{ g_i: 1\leq i \leq n-b\}.
$$

We first note that the ideal $J$ generated by the $f_{ij}$'s is the ideal of $2$-minors of the matrix of indeterminates
 $\left(  \begin{matrix} x_1 & x_2 & \dots & x_{n-1} \\  x_2 & x_3 &\dots & x_n  \end{matrix} \right)$ and it is the defining ideal of the rational normal scroll.
We may also view $J$ as the binomial edge ideal attached to the complete graph $K_n$. Since $K_n$ is a closed graph,
by \cite[Theorem 1.1]{CDE} we obtain that the $f_{ij}$'s form a Gr\"obner basis for $J$ with respect to our term order.
We refer to \cite{CDE} for the unexplained terminology.

Since the $S$-pair of two monomials is zero,  all that is left to show is that  $S(f_{ij}, g_k) \stackrel{\mathcal{G}}{\rightarrow} 0$.
For $1\leq i<j \leq n-1$ the leading monomial $\ini_<(f_{ij})= x_{i+1}x_j$ is coprime to $x_n$.
If $\gcd (\ini_<(f_{ij}), g_k)=1$, then  $S(f_{ij}, g_k) \stackrel{\mathcal{G}}{\rightarrow} 0$, see \cite[Proposition 2.15]{EH}.
Otherwise, if $\gcd( x_{i+1} x_j, g_k) \neq 1$, equivalently
$x_{b+k}|x_{i+1} x_j$ we get that $b+k \in \{i+1, j \}$ and $b+k<j+1$.
Therefore $S(f_{ij}, g_k)= x_n^a x_i x_{j+1}= x_i (x_n^a  x_{j+1})$ is a multiple of an element in $\mathcal{G}$ and
 $S(f_{ij}, g_k) \stackrel{\mathcal{G}}{\rightarrow} 0$  in this case, as well.
\end{proof}

The next statement describes the Koszul arithmetic sequences.

\begin{Proposition}
\label{prop:arithmetic}
Fix $n\geq 3$. Let $a_1$ and $d$ be positive integers such that  $n\leq a_1$ and $\gcd(a_1, d)=1$.
If we let
$$
H= \langle a_1, a_1+d, \dots, a_1+ (n-1)d \rangle,
$$
the following are equivalent:
\begin{enumerate}
\item[{\em (a)}] $I_H^*$ has a quadratic Gr\"obner basis;
\item[{\em (b)}] $\gr_\mm K[H]$ is Koszul;
\item[{\em (c)}] $I_H^*$ is generated by quadrics;
\item[{\em (d)}] $n\leq a_1 \leq  2n-2$.
\end{enumerate}
\end{Proposition}

\begin{proof}
With notation as in the proof of Proposition \ref{prop:arithm-gbasis} we observe that
 $I^*_H$ is generated by quadrics if and only if $a=1$, i.e. $n\leq a_1 \leq 2n-2$, hence (c) $\Rightarrow$ (d).
By Proposition \ref{prop:arithm-gbasis} in these cases   $I_H^*$ has a quadratic Gr\"{o}bner basis, hence (d)  $\Rightarrow$ (a).
The rest of the implications are known to hold in general.
\end{proof}

The  class of compound semigroups was recently introduced by  Kiers, O'Neill and Ponomarenko   in \cite{KOP}.
\begin{Definition}{\em (\cite{KOP})
Consider the integers $2\leq a_i<b_i$ such that $\gcd(a_i, b_i\cdots b_n)=1$ for $i=1, \dots, n$.
Let $q_i= b_1\cdots b_{i-1} a_{i} \cdots a_n$, for $i=1, \dots, n+1$.

The sequence $  q_1, \dots, q_{n+1}$ is   called a
{\em compound sequence} and the semigroup $H =\langle q_1, \dots, q_n \rangle$ is called a {\em compound semigroup}. }
\end{Definition}

With notation as above, if $a_1=\dots =a_n$ and $b_1=\dots =b_n$, then $q_1, \dots, q_{n+1}$ is a geometric sequence.
In what follows we show that for  any compound semigroup $H$ we have $I_H^*$ is CI and this allows to identify the quadratic (equivalently Koszul) compound semigroups.

\begin{Proposition}
\label{prop:compound}
  Let $2<a_i<b_i$  positive integers such that
	$\gcd(a_i, b_i\cdots b_n)=1$ for $i=1, \dots, n$.	
Let
$$
H= \langle a_1a_2\cdots a_n, \  b_1 a_2\cdots a_n, \ b_1 b_2a_3\cdots a_n, \dots,\  b_2\cdots b_n \rangle,
$$
  The following hold:
\begin{enumerate}
\item[{\em (a)}] the ideal $I_H^*$ is a complete intersection;
\item[{\em (b)}] $I_H^*$ is quadratic \iff $H$ is Koszul \iff $a_i=2$ for $i=1, \dots, n$.
\end{enumerate}
\end{Proposition}

\begin{proof}
For part (a) we observe that $H$ is obtained from another compound semigroup by a simple gluing:
$$
H=\langle a_1a_2\dots a_n, \ b_1\langle a_2\dots a_n, b_2a_3\dots a_n, \dots, b_2\dots b_n\rangle \ \rangle,
$$
which gives a first (gluing) relation $f_1=x_1^{b_1}-x_2^{a_1}$ in $I_H$. We continue decomposing into compound semigroups of smaller embedding dimension and
 in the end we get
\begin{equation}
\label{eq:toric-geometric}
I_H=(x_1^{b_1}-x_2^{a_1}, x_2^{b_2}-x_3^{a_2}, \dots, x_n^{b_n}-x_{n+1}^{a_n}).
\end{equation}
By Delorme's Theorem \ref{thm:Delorme} we get that $I_H$ is a complete intersection.

If  for any $f\in S=K[x_1, \dots, x_{n+1}]$ we denote by  $\bar{f}$ the polynomial $f(0, x_2, \dots, x_{n+1})$ in $K[x_2, \dots, x_{n+1}]$,
one has $\bar{I}_H= (x_2^{a_1}, x_3^{a_2}, \dots, x_{n+1}^{a_{n+1}})$. These generators form a standard basis
for the homogeneous ideal $\bar{I}_H$, and  each of these generators can be lifted to an element in $I_H$ with the same initial degree.
Therefore by the criterion   in \cite[Theorem 1]{He-reg} (see also \cite[Lemma 1.2]{HeS}) we conclude that the generators in
\eqref{eq:toric-geometric}
are a standard basis, hence
$I_H^*=(x_2^{a_1}, x_3^{a_2}, \dots, x_{n}^{a_{n+1}})$ and $I_H^*$ is a complete intersection.

From this it follows that $I_H^*$ is quadratic if and only if $a_1=\dots =a_n= 2$.
The remaining equivalence at  (b) is given by Theorem \ref{thm:quad-ci}.
\end{proof}

The next result shows a family  of CI's  that are never quadratic.

\begin{Proposition}
\label{prop:never-quad}
Let  $a_1, \dots a_n$ be pairwise coprime positive integers, $n>2$.   Let $P=\prod_{i=1}^n a_i$.
The numerical semigroup $H= \langle P/a_1, \dots, P/a_n\rangle$ is a complete intersection semigroup that is never quadratic.
\end{Proposition}
\begin{proof}
Without loss of generality assume $a_1>\dots >a_n$.
We prove the CI property by induction on $n$ and we identify a decomposition satisfying Delorme's Theorem \ref{thm:Delorme}.
Letting $Q= P/a_1$, we may write
$$
H= \langle P/a_1,\ a_1 \langle Q/a_2, \dots, Q/a_n \rangle \ \rangle,
$$
hence $H$ is obtained via a simple gluing from the semigroup $L= \langle Q/a_2, \dots, Q/a_n \rangle $ which is CI by the induction hypothesis.
This gluing gives $I_H= (x_1^{a_1}-x_2^{a_2}, I_L)$, and after iterating this un-gluing several times we obtain
$$
I_H=(x_1^{a_1}-x_2^{a_2}, x_2^{a_2}-x_3^{a_3}, \dots, x_{n-1}^{a_{n-1}}- x_n^{a_n}).
$$
We argue as in the proof of Proposition \ref{prop:compound}(b). Modulo $x_1$, $\bar{I}_H= (x_2^{a_2}, x_3^{a_3}, \dots,   x_n^{a_n})$
whose (monomial) generators are a standard basis and they may be naturally lifted to polynomials in $I_H$ with the same initial degree.
By the same criterion in  \cite[Lemma 1.2]{HeS}, $I_H$ is generated by a standard basis and $I_H^*= (x_2^{a_2}, x_3^{a_3}, \dots,   x_n^{a_n})$.
Since $n>2$ and the $a_i$'s are coprime, it is clear that $I_H^*$ is not generated in degree $2$.

{\em Second (partial) proof of the non-quadraticity:}
Knowing that $H$ is CI, if it were also quadratic, by Theorem \ref{thm:bounds} we had $P/a_1= 2^{n-1}$.
Since the $a_i$'s are coprime, $P/a_1=2^{n-1}$ has at least $n-1$ distinct prime divisors, which is false for $n>2$.
\end{proof}

\subsection{Koszul $3$-semigroups}

We next describe the  Koszul numerical semigroups of embedding dimension $3$.

Let $H$ be a quadratic numerical semigroup minimally generated by $a_1<a_2<a_3$. By Theorem \ref{thm:bounds}, $e(H)\in \{ 3,4 \}$.
For any of these two values, by Proposition \ref{prop:g-quadratic} we get that $H$ is Koszul.

Assume $e(H)=4$. By Theorem \ref{thm:bounds}(c), $H$ is a quadratic complete intersection, hence by Theorem \ref{thm:quad-ci} it
is obtained from $\NN$ via quadratic gluings:  $H=\langle 2 \langle 2, c\rangle, \ell  \rangle$,
where $c$ and $\ell$ are odd integers,  $c>1$, $\ell \in \langle 2, c \rangle \setminus \{2, c \}$, i.e. $\ell = 2\alpha + c \beta$, $\beta= 2\gamma +1$ is odd,
and $\alpha$ and $\gamma$ are not simultaneously equal to $0$.  Equivalently,
$$
H= \langle 4, 2c, \ell\rangle = \langle 4, 2c, 2\alpha+ c \beta \rangle = \langle 4, 2c, 2(\alpha+ c\gamma)+c \rangle= \langle 4, 2c, 2a+c \rangle,
$$
where $a,c$ are positive integers with  $c>1$   odd.
 Here we denoted $a=\alpha+c\gamma$. Clearly $a>0$, otherwise $H=\langle 4, c \rangle$ and $\embdim(H)<3$,  a contradiction.

We group these findings into the next result.
\begin{Proposition}
\label{prop:3-semi-quad}
Let $H$ a numerical semigroup with $\embdim(H)=3$.
The following are equivalent:
\begin{enumerate}
\item [{\em (a)}] $H$ is a Koszul semigroup;
\item [{\em (b)}] $H$ is a quadratic semigroup;
\item [{\em (c)}] $e(H)=3$ or $e(H)=4$ and $H =  \langle 4, 2c,  2a+c \rangle$, where $a,c$ are positive integers with  $c>1$   odd.
\end{enumerate}
\end{Proposition}

T.~Shibuta has communicated to the second author that he could prove Proposition \ref{prop:3-semi-quad}  using \cite{He-semi} and \cite{He-reg}.

\subsection{Special almost complete intersections}
Let $H$ be a numerical semigroup minimally generated by $a_1<\dots <a_n$, where $n>2$.

For $i=1, \dots, n$, let $c_i$ be the smallest positive integer such that $c_i a_i$ is a sum of the other generators.
This produces a binomial $f_i=x_i^{c_i}-m_i$ in $I_H$, where  $m_i$ is  not divisible by $x_i$.

It is clear that $x_i^{c_i}$ is the least pure power of $x_i$ that occurs as a term of any polynomial in $I_H$.
If we are able to choose $m_i$ not a pure power for any $i$, then the $f_i$'s are distinct.
If moreover they generate $I_H$ we say that  $H$ is a {\em special almost complete intersection semigroup}.

By \cite{He-semi}, any $3$-generated numerical semigroup that is not a complete intersection, is a special almost complete intersection.

We note that by Lemma \ref{lemma:quad-basis}, if $H$ is quadratic and a special almost complete intersection numerical semigroup, then $I^*_H$ is an almost complete intersection ideal.

\begin{Proposition}
Assume $H$ is a quadratic and special almost complete intersection semigroup.
Then  $I_H^*$ has a quadratic Gr\"obner basis with respect to the degree reverse lexicographic order  if and only if $e(H)=2^{n-1}-2^{n-3}$.
In particular, $\gr_\mm K[H]$ is Koszul if $e(H)=2^{n-1}-2^{n-3}$.
\end{Proposition}

\begin{proof}
With notation as above, if $H$ is quadratic, by Lemma \ref{lemma:quad-semi} we get that $ c_1>2$ and $c_i=2$ for all $i>1$.
By Lemma \ref{lemma:quad-basis} we obtain that $I_H^*=(f_1^*, \dots, f_n^*)$ and that is a minimal generating set.
Clearly,  we have
$\ini_<(f_i^*)=x_i^2$ for $1< i \leq n$ and $\ini_<(f_1^*)=x_j x_k$ for some $2\leq j< k<n$.
We note that $x_1$ does not occur in any of the $\ini_<(f_i^*)$, hence
\begin{eqnarray*}
e(H)=e(S/I^*_H)= e(S/\ini_<(I^*_H)) &\leq& \ell(K[x_2, \dots, x_n]/(\ini_<(f_1^*), \dots, \ini_<(f_n^*))) \\
                                 &=& \ell(K[x_2, \dots, x_n]/(x_2^2, \dots, x_n^2, x_j x_k)) \\
																&=& 2^{n-1}-2^{n-3},
\end{eqnarray*}
after a computation similar to the one in the proof of Theorem \ref{lemma:quad-basis}(b).
We conclude that $e(H)=2^{n-1}-2^{n-3}$ if and only if $f_1^*, \dots, f_n^*$ form a Gr\"obner basis for $I_H^*$.
\end{proof}

\begin{Example}
\label{eq:saci}
{\em
A quadratic special almost complete intersection of embedding dimension $n$ need not to be Koszul if $e(H)< 2^{n-1}-2^{n-3}$.

\noindent Indeed, let $H=\langle 11,13,14,15,19 \rangle$. Then
$$
I_H=(x_1^3-x_3x_5, x_2^2-x_1x_4, x_3^2-x_2 x_4, x_4^2-x_1 x_5, x_5^2-x_1x_2x_3).
$$
As noticed in Remark \ref{rem:aci-nonkoszul}, $H$ is quadratic, but it is not a Koszul semigroup. }
\end{Example}

As an extension of  Corollary \ref{cor:ci-aci} we have the following.
\begin{Corollary}
\label{cor:more-saci}
Let $L$ be a special almost complete intersection numerical semigroup that is quadratic, and $\ell \in L\setminus G(L)$ an odd integer that is not a multiple of $e(L)$.
Then $H=\langle 2L, \ell \rangle$ is a special almost complete intersection semigroup, too.
\end{Corollary}

\begin{proof}
By Corollary \ref{cor:ci-aci}, together with Lemma \ref{lemma:quad-basis} we get that $H$ is quadratic and
that $I_H$ is an almost complete intersection ideal.
Let $n=\embdim(H)$. If we denote $f$ the gluing relation, using the convention that $x_n$ corresponds to the new  generator $\ell$, we have that
$I_H= (I_L, f)$, and  the gluing relation $f$ from Eq. \eqref{eq:glue-poly} is of the form $x_n^2-m$, where $m$ is a monomial in the variables $x_1, \dots, x_{n-1}$.

We claim that we may choose $f$ such that $m$ is not a pure power.
Indeed, if $m=x_1^c$ with $c>1$, then $\ell$ is a multiple of $e(L)$, which contradicts our assumption.
Assume $m=x_i^c$ with $1< i \leq n-1$ and $c>1$.
By our assumption on $L$ and Lemma  \ref{lemma:quad-semi}, there exists an equation $f_i=x_i^2-u$, where $u$ is a monomial which is not a pure power.
Then we can replace $f$ by $x_n^2-x_i^{c-2}u$.
Since $L$ was special almost CI, we conclude that the same is true about $H$.
\end{proof}

\begin{Remark}
{\em
As a consequence of Corollary \ref{cor:more-saci}, starting from any quadratic special almost complete intersection semigroup, by gluing
 we can construct semigroups with these properties of any larger embedding dimension.
}
\end{Remark}

\subsection{Symmetric and pseudo-symmetric Koszul 4-semigroups}
\label{ssec:symmetric}
The {\em pseudo-Frobenius numbers} of the numerical semigroup $H$ are the elements of the finite set
$$
PF(H)= \{ n\in \ZZ \setminus H: n+h\in H, \text{ for all } h\in H\setminus\{0\} \}.
$$
The {\em Frobenius number} of $H$, usually defined as $g(H)= \max \NN\setminus H$, also satisfies
 $g(H) = \max PF(H)$.

The semigroup  $H$ is called {\em symmetric} if for any integer  $n$  exactly one of $n$ and $g(H)-n$ is in $H$.
Algebraically, by a celebrated theorem of Kunz (\cite{Kunz}), $H$ is symmetric if and only if $K[H]$ is a Gorenstein ring.
One can check that $H$ is symmetric if and only if $PF(H)=\{ g(H) \}$.

The semigroup $H$ is called {\em pseudo-symmetric} if $PF(H)= \{ g(H)/2, g(H) \}$.

In the sequel we describe the $4$-generated symmetric or pseudo-symmetric numerical semigroups that are also Koszul.

\medskip
\subsubsection{The symmetric case}
  Let $H$ be a symmetric numerical semigroup such that $\embdim(H)=4$.
	
If $H$ is CI and Koszul, by Theorem \ref{thm:quad-ci} we have that $H$ is obtained from $\NN$ by a sequence of quadratic simple gluings. Using also Proposition \ref{prop:3-semi-quad} we have that
$H=\langle 2 \langle 4, 2c, 2a+c\rangle, \ell \rangle= \langle 8, 4c, 4a+ 2c, \ell \rangle$,
where $a, c, \ell$ are positive integers, $c,\ell > 1$ are odd, and $\ell \in \langle 4,2c, 2a+c \rangle \setminus \{2a+c\}$.

If $H$ is not CI, we employ the following characterization found by Bresinsky \cite{Bres-gore}, as given by Barucci et al. in \cite[Theorem 3]{BFS}.

\begin{Theorem} (Bresinsky, \cite[Theorem 5, Theorem 3]{Bres-gore})
\label{thm:Bresinsky} The numerical semigroup $H=\langle a_1, a_2, a_3, a_4 \rangle$ is $4$-generated symmetric, not a complete intersection, if and only if there are integers
$c_i$, $1 \leq i \leq 4$, $\alpha_{ij}, ij\in \{ 21, 31, 32, 42, 13, 43, 14, 24 \}$, such that $0< \alpha_{ij}< c_i$, for all $i$,$j$,
\begin{eqnarray*}
c_1 = \alpha_{21}+ \alpha_{31},\ c_2= \alpha_{32}+ \alpha_{42},\ c_3=\alpha_{13}+\alpha_{43},\  c_4=\alpha_{14}+ \alpha_{24}, \\
a_1 = c_2 c_3 \alpha_{14}+ \alpha_{32}\alpha_{13}\alpha_{24}, \quad a_2= c_3 c_4 \alpha_{21}+ \alpha_{31}\alpha_{43}\alpha_{24}, \quad \\
a_3 = c_1 c_4 \alpha_{32}+ \alpha_{14}\alpha_{42}\alpha_{31}, \quad a_4= c_1 c_2 \alpha_{43}+ \alpha_{42}\alpha_{21}\alpha_{13}. \quad
\end{eqnarray*}
Then $K[H]\cong S/(f_1, f_2, f_3, f_4, f_5)$, where
\begin{eqnarray*}
f_1 = x_1^{c_1}- x_3^{\alpha_{13}}x_4^{\alpha_{14}},\quad f_2= x_2^{c_2}-x_1^{\alpha_{21}}x_4^{\alpha_{24}}, \quad f_3=x_3^{c_3}- x_1^{\alpha_{31}}x_2^{\alpha_{32}},\\
f_4 = x_4^{c_4}-x_2^{\alpha_{42}}x_3^{\alpha_{43}}, \quad f_5= x_3^{\alpha_{43}}x_1^{\alpha_{21}}-x_2^{\alpha_{32}}x_4^{\alpha_{14}}.\quad\quad\quad
\end{eqnarray*}
\end{Theorem}

For quadratic, symmetric and not CI semigroups we obtain the following classification result.

\begin{Theorem}
\label{thm:4-quadr-gore}
Let $H$ be a $4$-generated  semigroup that is symmetric and not a complete intersection. The following are equivalent:
\begin{enumerate}
\item[{\em (a)}] $H$ is Koszul;
\item[{\em (b)}] $H$ is quadratic;
\item[{\em (c)}] $e(H)=5$;
\item[{\em (d)}] $H= \langle 5, 4a+b, 2a+3b, 3a+2b \rangle$ for some  positive integers $a, b$ such that $a-b$ is not divisible by $5$.
\end{enumerate}

Moreover, the integers $a$ and $b$ in {\em (d)} are uniquely determined by $H$.
\end{Theorem}

\begin{proof}
The implication
  (a) \implies (b)   is clear.
For   (b) \implies (c) assume  $H= \langle a_1, a_2, a_3, a_4 \rangle$ is quadratic.
Using Lemma \ref{lemma:quad-semi} and Bresinsky's Theorem \ref{thm:Bresinsky}, without loss of generality we may  assume that $c_1>2$ (i.e. $e(H)= a_1$) and $c_2=c_3=c_4=2$.
The conditions $0< \alpha_{ij}<c_i$ give $\alpha_{ij}=1$ for $ij \in \{    32, 42, 13, 43, 14, 24 \}$, hence $a_1 = c_2 c_3 \alpha_{14}+ \alpha_{32}\alpha_{13}\alpha_{24}=5$.

For  (c) \implies (d), taking into account the restrictions in Theorem \ref{thm:Bresinsky} we note that the equation
$$
5 = c_2 c_3 \alpha_{14}+ \alpha_{32}\alpha_{13}\alpha_{24}
$$
 holds only if $\alpha_{ij}=1$ for $ij \in \{14, 32, 13, 24 \} $ and if $c_2=c_3=2$. The latter set of equalities yield   $\alpha_{42}=\alpha_{43}=1$. For brevity we denote $a=\alpha_{21}$ and $b=\alpha_{31}$.
We plug  these values into  Bresinsky's Theorem and we get $a_2= 4a+b$, $a_3=2a+3b$, $a_4=3a+2b$.

We show that this parametrization is one-to-one.
Let $a,b,a',b'>0$ and $5\nmid a-b$, $5\nmid a'-b'$ such that
$$
\langle 5, 4a+b, 2a+3b, 3a+2b \rangle = \langle 5, 4a'+b, 2a'+3b', 3a'+2b' \rangle.
$$
Note that $4a+b, 3a+2b, 2a+3b$  and $4a'+b', 3a'+2b', 2a'+3b'$ are arithmetic sequences with common difference $b-a$, and $b'-a'$ respectively.

If $(b-a)(b'-a') <0 $, then  $b-a= a'-b'$ and $5a+(b-a)=4a+b= 2a'+3b'= 5b'+ 2(a'-b')$, hence $5| b-a$, which is false.

If $(b-a)(b'-a') >0$, then $b-a= b'-a'$ and $5a+ (b-a)=4a+b = 4a'+b'= 5a'+ (b'-a')$, hence $(a,b)=(a',b')$ and we are done.

For     (d) \implies (a) we first note using Bresinsky's Theorem  that $H$ is indeed symmetric and all
the $c_i$'s and the $\alpha_{ij}$'s can be read from the proof of the implication (c)\implies (d).

 Consequently $I_H=(f_1, f_2, f_3, f_4, f_5)$, where
\begin{eqnarray*}
f_1=x_1^{a+b}-x_3 x_4, \quad f_2=x_2^2-x_1^a x_4, \quad f_3=x_3^2-x_1^b x_2, \\ f_4=x_4^2-x_2x_3, \quad f_5=x_3x_1^a-x_2 x_4.\quad\quad\quad\quad
\end{eqnarray*}
Modulo $x_1$ we get
\begin{equation}
\bar{I}_H=(x_3x_4, x_2^2, x_3^2, x_4^2-x_2x_3, x_2x_4),
\end{equation}
a monomial ideal whose generators (at the same time  a standard basis) may be lifted to the $f_i$'s in $I_H$ and keeping the same initial degree.
We apply the criterion in  \cite[Lemma 1.2]{HeS} to conclude the $x_1$ is regular on $S/I^*_H$ and $f_1, \dots, f_5$ form a standard basis for $I_H$.
Hence
\begin{equation}
\label{eq:smart}
I_H^*= \begin{cases} (x_3x_4, x_2^2, x_3^2, x_4^2-x_2x_3, x_2x_4), \text{ if } a\neq 1 \text{ and } b\neq 1,  \\
  (x_3x_4, x_2^2-x_1 x_4, x_3^2, x_4^2-x_2x_3, x_3x_1-x_2x_4), \text{ if } a= 1, \text{ and } b\neq 1,  \\
  (x_3x_4, x_2^2, x_3^2-x_1x_2, x_4^2-x_2x_3, x_2x_4), \text{ if } a\neq 1 \text{ and } b = 1.
\end{cases}
\end{equation}
It is easy to check that in each of these situations  $I_H^*$ has a quadratic Gr\"obner basis with respect to 
the degree reverse lexicographic order induced by $x_4>x_3>x_2>x_1$, in particular $S/I_H^*$ is a Koszul ring.
\end{proof}

We verified with Singular (\cite{Sing}) that in any of the three   cases  \eqref{eq:smart} the ring  $S/I^*_H$ is Gorenstein.
Together with Theorem \ref{thm:bounds}(c) we obtain the following.

\begin{Corollary}
\label{cor:strict-gore}
Let $H$ be a $4$-generated symmetric and quadratic numerical semigroup. Then $\gr_\mm K[H]$ is Gorenstein.
\end{Corollary}

\subsubsection{The pseudosymmetric case}

Four-generated pseudo-symmetric semigroups were characterized by Komeda in  \cite{Komeda},
where these were studied under the name almost symmetric. In the formulation from \cite{BFS}, the following holds.

\begin{Theorem} (Komeda, \cite[Theorems 6.4, 6.5]{Komeda})
\label{thm:Komeda}
The semigroup $H= \langle a_1, a_2, a_3, a_4 \rangle$ is pseudo-symmetric if and only if
there are positive integers $c_i>1 , 1\leq i\leq 4$, and $0<\alpha_{21}< c_1-1$, such that
\begin{gather*}
a_1= c_2 c_3 (c_4-1)+1, \quad a_2=\alpha_{21} c_3 c_4+ (c_1-\alpha_{21}-1)(c_3-1)+ c_3,\\
a_3= c_1c_4+ (c_1-\alpha_{21}-1)(c_2- 1)(c_4-1)-c_4+1, \\a_4=c_1 c_2(c_3-1)+ \alpha_{21}(c_2-1)+ c_2.
\end{gather*}
and $\gcd(a_1, a_2, a_3, a_4)=1$.

Then $K[H] \cong S/(f_1, f_2, f_3, f_4, f_5)$, where
\begin{gather*}
f_1= x_1^{c_1}- x_3x_4^{c_4-1}, \quad f_2=x_2^{c_2}- x_1^{\alpha_{21}}x_4, \quad f_3=x_3^{c_3}-x_1^{c_1-\alpha_{21}-1}x_2,\\
f_4=x_4^{c_4}- x_1 x_2^{c_2-1}x_3^{c_3-1}, \quad f_5=x_3^{c_3-1}x_1^{\alpha_{21}+1}-x_2 x_4^{c_4-1}.
\end{gather*}
\end{Theorem}

The quadratic pseudo-symmetric $4$-semigroups are described by the following result.

\begin{Proposition}
Let $H=\langle a_1, a_2, a_3,a_4\rangle$ be  a pseudo-symmetric numerical semigroup.
The following are equivalent:
\begin{enumerate}
\item [{\em (a)}] $H$ is Koszul;
\item [{\em (b)}] $H$ is quadratic;
\item [{\em (c)}] $H=\langle 5, 3 a+ b+1, 3b-a-2, a+2b+2\rangle$ for some integers $0<a< b-1$ such that $3a+b+1$ is not a multiple of  $5$.
\end{enumerate}
\end{Proposition}

\begin{proof}
For   (b)\implies (c), by Lemma \ref{lemma:quad-semi} and the restriction $0<\alpha_{21}<c_1-1$ in Komeda's Theorem \ref{thm:Komeda},
we get that $e(H)=a_1$ and $c_2=c_3=c_4=2$.
This gives
$a_1=5$, $a_2= 3 \alpha_{21}+ c_1+1$, $a_3= 3c_1-\alpha_{21}-2$, $a_4= 2c_1+ \alpha_{21}+2$.
Letting $a=\alpha_{21}$ and $b=c_1$ we obtain  the desired description.

Note that $3 a_2 \equiv a_3 \mod  5$ and $2 a_2 \equiv a_4 \mod  5$. Therefore $\gcd(a_1, a_2, a_3, a_4)=1$ precisely when $3a+b+1$ is not a multiple of  $5$.

For  (c) \implies (a), by Komeda's theorem we have
$$
I_H= (x_1^{b}-x_3 x_4,\ x_2^2-x_1^{a}x_4,\ x_3^2-x_1^{b-a-1}x_2,\ x_4^2-x_1x_2x_3,\ x_3 x_1^{a+1}-x_2 x_4).
$$
Modulo $x_1$ it becomes
$\bar{I}_H=( x_3 x_4, x_2^2, x_3^2, x_4^2, x_2 x_4)$.
These monomials can be lifted to polynomials in  $I_H$ with the same initial degree, hence by using the criterion in \cite[Lemma 1.2]{HeS},
$I_H$ is generated by a standard basis and $x_1$ is a regular element (of degree 1) on $S/I_H^*$.
From here we notice that $I_H^*$ is generated in degree $2$.
Since $(S/I_H^*)/x_1(S/I_H^*)\cong K[x_2,x_3,x_4]/( x_3 x_4, x_2^2, x_3^2, x_4^2, x_2 x_4)$  which is Koszul,
by \cite[Lemma 2]{BF-poincare},  $S/I^*_H$ is Koszul, as well.
\end{proof}

\medskip
{\bf Acknowledgement}.
The use of CoCoA \cite{Cocoa} and Singular \cite{Sing} was vital for the development of this paper. We wish to thank their respective teams of developers.

We thank M.E.~Rossi and G.~Valla for pointing us towards their work in \cite{RV} which gives evidence to support our Question \ref{que:forbidden-aci},
and we also thank M.E.~Rossi for suggesting Remark \ref{rem:cdnr}.

The second author was  supported by  a grant of the
Romanian Ministry of Education, CNCS--UEFISCDI, project number PN-II-RU-PD-2012-3--0656.
Important progress on this research took place  during the authors' visits to the co-author's home institution.
We are grateful for their hospitality.

{}

\end{document}